\documentclass[11pt,reqno]{amsart}%
\usepackage{graphicx}
\usepackage{amsmath}
\usepackage{amsfonts}
\usepackage{amssymb}
\usepackage{caption}
\usepackage{amsthm}
\usepackage{enumerate}
\usepackage{url}
\usepackage{cite}
\usepackage{color}
\usepackage{tikz}
\usepackage{mathrsfs}
\usepackage[all]{xypic}
\usepackage[margin=1in]{geometry}%
\providecommand{\U}[1]{\protect\rule{.1in}{.1in}}
\newtheorem{theorem}{Theorem}[section]
\newtheorem{proposition}[theorem]{Proposition}
\newtheorem{lemma}[theorem]{Lemma}
\newtheorem{corollary}[theorem]{Corollary}
\theoremstyle{definition}

\begin{document}
\title{Relations of the Nuclear Norms of a Tensor and its Matrix Flattenings}
\author{Shenglong~Hu}\thanks{This work is partially supported by National Science Foundation of China (Grant No. 11401428). }
\address{Department of Mathematics, School of Science, Tianjin University, Tianjin, China.}
\email{timhu@tju.edu.cn}

\address{Department of Mathematics, National University of Singapore}
\email{mathush@nus.edu.sg}

\begin{abstract}
For a $3$-tensor of dimensions $I_1\times I_2\times I_3$, we show that the nuclear norm of its every matrix flattening is a lower bound of the tensor nuclear norm, and which in turn is upper bounded by $\sqrt{\min\{I_i : i\neq j\}}$ times the nuclear norm of the matrix flattening in mode $j$ for all $j=1,2,3$. The results can be generalized to $N$-tensors with any $N\geq 3$. Both the lower and upper bounds for the tensor nuclear norm are sharp in the case $N=3$. A computable criterion for the lower bound being tight is given as well.
\end{abstract}
\maketitle

\section{Introduction}\label{sec:intro}
The fundamental significancy of the matrix nuclear norm is commonly admitted, in both theory and applications, especially in matrix completion problems, see \cite{horn-johnson,golub-vanloan,candes-recht} and references therein. 
Likewise, the tensor nuclear norm has been recognized to be of great interesting and importance very recently \cite{friedland-lim,yuan-zhang,derksen,lim13}.  

Though much effort in developing theory and tools for handling tensors in recent years, compared with those for matrices, they are still in infancy \cite{landsberg,lim13}. As a result,  
in the very important problem of tensor completion, nuclear norms of the matrix flattenings of the underlying tensor are popularly used as alternatives of the less explored tensor nuclear norm \cite{gandy-recht-yamada}. However, in very recently, it is shown that the usage of tensor nuclear norm can gain drastically smaller sample size to guarantee exact recovery of lower rank tensors in large dimensions \cite{yuan-zhang}. Therefore, it would be of interesting to know some relationships between the two approaches of the tensor completion problem. Especially, the practical powerful approach through matrix flattenings suggests that there should be closely related relationships between the tensor nuclear norm, which possesses stronger theoretical recovery property, and its matrix flattening nuclear norms. 

On the other side, it is shown also very recently that the computation of the tensor nuclear norm is NP-hard \cite{friedland-lim}, which implies that the tensor completion problem based on tensor nuclear norm is also NP-hard. So, it is interesting to get approximations of the tensor nuclear norm with known worst case bounds. From the approaches of tensor completion problem, the nuclear norms of the matrix flattenings would be the first choices.  

This article establishes some relationships between them. Thus, it provides a rationale for the usage of nuclear norms of the matrix flattenings in tensor completion from a new perspective, and also computable bounds for the NP-hard tensor nuclear norm. 

We will first focus on third order tensors (or $3$-tensors) in Sections~\ref{sec:nuclear}, \ref{sec:flatten} and \ref{sec:bound}, and then extend the results to tensors of higher orders in the last section (cf.\ Section~\ref{sec:general}). 

\section{Tensor Nuclear Norm}\label{sec:nuclear}
Let $\mathbb R^{I\times J\times K}$ be space of third order tensors (or $3$-tensors) of dimensions $I\times J\times K$ with entries in the field of real numbers. A tensor $\mathcal A\in\mathbb R^{I\times J\times K}$ consists of $IJK$ entries $a_{ijk}$ with $i\in\{1,\dots,I\}$, $j\in\{1,\dots,J\}$ and $k\in\{1,\dots,K\}$. 
Associated with the tensor space $\mathbb R^{I\times J\times K}$ are the natural inner product:
\[
\langle\mathcal A,\mathcal B\rangle:=\sum_{1\leq i\leq I}\sum_{1\leq j\leq J}\sum_{1\leq k\leq K}a_{ijk} b_{ijk} ,
\]
and the induced norm
\[
\|\mathcal A\|_{\operatorname{HS}}:=\sqrt{\langle\mathcal A,\mathcal A\rangle},
\]
which is referred as the Hilbert-Schmidt norm in the literature \cite{lim13}. Note that when $\mathcal A$ degenerates (i.e., $\min\{I,J,K\}=1$), the Hilbert-Schmidt norm reduces to the Frobenius norm of a matrix or the Euclidean norm of a vector. 

Tensors are generalizations of matrices. Two of the most important norms of a matrix are the spectral norm and its dual (i.e., the nuclear norm) \cite{horn-johnson, golub-vanloan}. Likewise, we can define the \textit{spectral norm of a tensor} $\mathcal A\in\mathbb R^{I\times J\times K}$ as
\begin{equation}\label{sepctralnorm}
\|\mathcal A\|:=\max\big\{\langle\mathcal A,\mathbf x\otimes\mathbf y\otimes\mathbf z\rangle : 
\mathbf x\in\mathbb R^I,\ \mathbf y\in\mathbb R^J,\ \mathbf z\in\mathbb R^K,\ 
\|\mathbf x\|=\|\mathbf y\|=\|\mathbf z\|=1\big\}.
\end{equation}
Here $\mathbf x\otimes\mathbf y\otimes\mathbf z\in\mathbb R^{I\times J\times K}$ is a rank one tensor with its $ijk$th entry being $x_iy_jz_k$.
Conveniently, we use the same $\|\cdot\|$ to mean the spectral norm of a tensor (which is denoted in calligraphic letter) and a matrix (which is denoted in capital letter), and the Euclidean norm of a vector (which is denoted in bold lower case letter). 

It is clear that $\|\cdot\|$ defines a norm over $\mathbb R^{I\times J\times K}$. The dual norm of $\|\cdot\|$ is defined as
\begin{equation}\label{nuclearnorm}
\|\mathcal A\|_\ast:=\max\big\{\langle\mathcal A,\mathcal B\rangle : \mathcal B\in\mathbb R^{I\times J\times K},\ \|\mathcal B\|=1\big\}.
\end{equation}
It can be proved that $\|\cdot\|_\ast$ is also a norm over $\mathbb R^{I\times J\times K}$. From the definitions, we see that the spectral norm and its dual norm of a tensor are generalizations of the spectral norm and the nuclear norm of a matrix respectively. 
We call $\|\mathcal A\|_\ast$ the \textit{nuclear norm} of the tensor $\mathcal A$. 

It is a fact that every tensor $\mathcal A\in\mathbb R^{I\times J\times K}$ can be decomposed into a sum of rank one tensors \cite{lim13,landsberg}:
\[
\mathcal A=\sum_{s=1}^r\lambda_s\mathbf x_s\otimes\mathbf y_s\otimes\mathbf z_s,
\]
with $\lambda_s\in\mathbb R$, and unit vectors $\mathbf x_s\in\mathbb R^I$, $\mathbf y_s\in\mathbb R^J$, and $\mathbf z_s\in\mathbb R^K$. 
It can be shown that \cite{lim13,friedland-lim}
\begin{equation}\label{decomp}
\|\mathcal A\|_\ast=\min\bigg\{\sum_{s=1}^r|\lambda_s| : \mathcal A=\sum_{s=1}^r\lambda_s\mathbf x_s\otimes\mathbf y_s\otimes\mathbf z_s, \|\mathbf x_s\|=\|\mathbf y_s\|=\|\mathbf z_s\|=1,\ s=1,\dots,r\bigg\}. 
\end{equation}
Note that the matrix nuclear norm has a similar characterization, i.e., the singular value decomposition. 

It is well-known that both the spectral norm and the nuclear norm of a matrix can be computed out very efficiently, in polynomial time complexity up to the machine accuracy \cite{golub-vanloan}. However, both the spectral norm and the nuclear norm of a tensor are successively proven to be NP-hard to compute \cite{hillar-lim,friedland-lim}. Despite the general NP-hardness, the nuclear norms of some special tensors can be determined, see \cite{lim13,friedland-lim,derksen}. 

\section{Matrix Flattening}\label{sec:flatten}
Given a $3$-tensor $\mathcal A\in\mathbb R^{I\times J\times K}$, we can regard it as a collection of $I$-vectors ($J$-vectors, $K$-vectors) $\mathbf a_{\cdot jk}$'s ($\mathbf a_{i\cdot k}$'s, $\mathbf a_{ij\cdot}$'s respectively), where for example 
\[
\mathbf a_{\cdot jk}=(a_{1jk},\dots,a_{Ijk})^\mathsf{T}\in\mathbb R^I.
\]

Let us focus on the $I$-vectors for a moment. There are altogether $JK$ $I$-vectors, and they are denoted by $\mathbf a_{\cdot jk}$ for $j=1,\dots,J$ and $k=1,\dots,K$. If we list all of them into a $I\times JK$ matrix with respect to a prefixed order of the set $\{(j,k) : 1\leq j\leq J,\ 1\leq k\leq K\}$ (eg.\ lexicographic order) as 
\[
A_{(1)}:=\big[\mathbf a_{\cdot 11},\dots,\mathbf a_{\cdot JK}\big],
\]
the resulting matrix is called the \textit{matrix flattening of the tensor $\mathcal A$ in mode $1$}. Similarly, we have the matrix flattenings $A_{(2)}$ and $A_{(3)}$ of the tensor $\mathcal A$ in mode $2$ and mode $3$ respectively.

The next lemma is immediate. 
\begin{lemma}[Isomorphism]\label{lem:corres}
For a fixed order of the set $\{(j,k) : 1\leq j\leq J,\ 1\leq k\leq K\}$, there is a one to one correspondence between the space $\mathbb R^{I\times J\times K}$ of $3$-tensors and the space of $\mathbb R^{I\times JK}$ of matrices through the matrix flattening in mode $1$. Similar results hold for mode $2$ and $3$. 
\end{lemma}

Therefore, we would like to fix an order of the set $\{(j,k) : 1\leq j\leq J,\ 1\leq k\leq K\}$, say the lexicographic order. Then, under the matrix flattening in mode $1$, the unique matrix associated to a tensor $\mathcal A$ is denoted by $A_{(1)}$ as before; while the unique tensor associated to a matrix $A\in\mathbb R^{I\times JK}$ is denoted by $\operatorname{ten_1}(A)$.

\section{Bounds from Matrix Flattening Nuclear Norms}\label{sec:bound}
As the nuclear norm of a tensor is NP-hard to compute, whereas the nuclear norms of matrices are easy to compute in any given accuracy, it becomes popular in applications, such as tensor completion \cite{gandy-recht-yamada}, to use
\begin{equation}\label{norm-av}
\|\mathcal A\|_{\#}:=\frac{1}{3}\big(\|A_{(1)}\|_\ast+\|A_{(2)}\|_\ast+\|A_{(3)}\|_\ast\big)
\end{equation}
or some other functionals over $(\|A_{(1)}\|_\ast,\|A_{(2)}\|_\ast,\|A_{(3)}\|_\ast)^\mathsf{T}$
as alternatives for $\|\mathcal A\|_\ast$.

We first show that every matrix flattening nuclear norm is a lower bound of the tensor nuclear norm. 
\begin{proposition}[Lower Bound]\label{prop:lowerbound}
For any $3$-tensor $\mathcal A\in\mathbb R^{I\times J\times K}$, we have
\[
\|A_{(i)}\|_\ast\leq \|\mathcal A\|_\ast,\ \text{for all }i=1,2,3.
\]
Therefore,
\begin{equation}\label{lowerbound}
\|\mathcal A\|_{\#}\leq\|\mathcal A\|_\ast. 
\end{equation}
\end{proposition}

\begin{proof}
We prove the case $\|A_{(1)}\|_\ast\leq\|\mathcal A\|_\ast$ and the others follow similarly. Moreover, \eqref{lowerbound} is an immediate consequence of these inequalities. 

Let 
\[
\mathfrak F:=\{\mathbf u\otimes\mathbf v\otimes\mathbf w : \mathbf u\in\mathbb R^I,\ \mathbf v\in\mathbb R^J,\ \mathbf w\in\mathbb R^K,\ \|\mathbf u\|=\|\mathbf v\|=\|\mathbf w\|=1\}
\]
be the set of all rank one tensors of unit length. 
Let
\[
\mathfrak D:=\{\mathcal U\in\mathbb R^{I\times J\times K} : \|\mathcal U\|\leq 1\}
\]
be the set of all tensors of length smaller than one. 

Likewise, let
\[
\mathfrak M:=\{\mathbf u\otimes\mathbf z : \mathbf u\in\mathbb R^I,\ \mathbf z\in\mathbb R^{JK},\ \|\mathbf u\|=\|\mathbf z\|=1\}
\]
be the set of all rank one matrices of unit length, and
\[
\mathfrak E:=\{ V\in\mathbb R^{I\times JK} : \|V\|\leq 1\}
\]
be the set of matrices of length smaller than one. 

It is easy to see that under the isomorphism with the matrix flattening in mode $1$ (cf.\ Lemma~\ref{lem:corres}),
\[
\mathfrak F\subset\mathfrak M,
\]
since 
\[
\|\mathbf v\otimes\mathbf w\|=\|\mathbf v\|\|\mathbf w\|.
\]
For a matrix $V\in\mathbb R^{I\times JK}$, we have
\[
\|V\|= \max\{\langle V,\mathbf u\otimes\mathbf z\rangle : \mathbf u\otimes\mathbf z\in\mathfrak M\}. 
\]
Therefore,
\[
\|V\|\geq\max\{\langle \operatorname{ten_1}(V),\mathbf u\otimes\mathbf v\otimes \mathbf w\rangle : \mathbf u\otimes\mathbf v\otimes\mathbf w\in\mathfrak F\}=\|\operatorname{ten_1}(V)\|.
\]
Thus, under the isomorphism with the matrix flattening in mode $1$,
\[
\mathfrak E\subset\mathfrak D.
\]
On the other side, by the definition of nuclear norm, it follows that
\[
\|A_{(1)}\|_\ast=\max\{\langle A_{(1)},V\rangle : V\in\mathfrak E\}. 
\]
Henceforth, these, together with 
\[
\langle A_{(1)},V\rangle=\langle\mathcal A,\operatorname{ten_1}(V)\rangle,
\]
imply that
\[
\|A_{(1)}\|_\ast\leq \max\{\langle\mathcal A,\operatorname{ten_1}(V)\rangle, \operatorname{ten_1}(V)\in\mathfrak D\}=\|\mathcal A\|_\ast. 
\]
\end{proof}

Given a tensor $\mathcal A\in\mathbb R^{I\times J\times K}$, let 
\[
\sum_{i=1}^r\sigma_i\mathbf x_i\otimes \mathbf z_i
\]
be the singular value decomposition of the matrix $A_{(1)}$, see \cite{horn-johnson,golub-vanloan}. Then,
\[
\|A_{(1)}\|_\ast=\sum_{i=1}^r\sigma_i,
\]
and both 
\[
\{\mathbf x_1,\dots,\mathbf x_r\}\ \text{and }\{\mathbf z_1,\dots,\mathbf z_r\}
\]
are orthonormal. 

Under the order of the set $\{(j,k) : 1\leq j\leq J,\ 1\leq k\leq K\}$, we can reformulate the vectors $\mathbf z_i$'s as 
$J\times K$ matrices $Z_i$'s. Then,
\begin{equation}\label{hs-norm}
\|Z_i\|_{\operatorname{HS}}=\|\mathbf z_i\|=1,\ \text{for all }i=1,\dots,r. 
\end{equation}
Define 
\begin{equation}\label{sum-nuclear}
\|\vee Z_i\|_\ast:=\max\{\|Z_i\|_\ast  : 1\leq i\leq r\}.
\end{equation}
Then, we have the following result.
\begin{proposition}[Upper Bound]\label{prop:svd-nuclear}
For any $3$-tensor $\mathcal A\in\mathbb R^{I\times J\times K}$, let $A_{(1)}$ be its matrix flattening in mode $1$ and
\[
\sum_{i=1}^r\sigma_i\mathbf x_i\otimes \mathbf z_i
\]
be the singular value decomposition of $A_{(1)}$. Let $Z_i$ be the matrix reformulation of $\mathbf z_i$ as above, then
\begin{equation}\label{bound-nuclear}
\|\mathcal A\|_\ast\leq \sum_{i=1}^r\sigma_i\|Z_i\|_\ast\leq \|A_{(1)}\|_\ast\|\vee Z_i\|_\ast.
\end{equation}
Similar results hold for matrix flattenings in mode $2$ and mode $3$. 
\end{proposition}

\begin{proof}
Let 
\[
Z_i=\sum_{j=1}^{r_i}\mu_{i,j}\mathbf v_{i,j}\otimes\mathbf w_{i,j}
\]
be the singular value decomposition of the matrix $Z_i$ for all $i=1,\dots,r$. Then,
\[
\|Z_i\|_\ast=\sum_{j=1}^{r_i}\mu_{i,j}.
\]
Since 
\[
A_{(1)}=\sum_{i=1}^r\sigma_i\mathbf x_i\otimes \mathbf z_i,
\]
under the isomorphism with the matrix flattening in mode $1$ (cf.\ Lemma~\ref{lem:corres}), we have that
\[
\mathcal A=\sum_{i=1}^r\sigma_i\mathbf x_i\otimes(\sum_{j=1}^{r_i}\mu_{i,j}\mathbf v_{i,j}\otimes\mathbf w_{i,j}).
\]
It follows from the characterization \eqref{decomp} that
\[
\|\mathcal A\|_\ast\leq \sum_{i=1}^r\sigma_i\big(\sum_{j=1}^{r_i}\mu_{i,j}\big)=\sum_{i=1}^r\sigma_i\|Z_i\|_\ast\leq \|\vee Z_i\|_\ast (\sum_{i=1}^r\sigma_i)=\|A_{(1)}\|_\ast\|\vee Z_i\|_\ast.
\]
Therefore, the result follows. 
\end{proof}

\begin{corollary}[A Criterion]\label{cor:nuclear}
Let $3$-tensor $\mathcal A\in\mathbb R^{I\times J\times K}$. If all the matrices $Z_i$'s as above have nuclear norm $1$, then
\[
\|\mathcal A\|_\ast=\|A_{(1)}\|_\ast.
\]
In this case, 
\begin{equation}\label{diagonaldec}
\mathcal A=\sum_{i=1}^r\sigma_i\mathbf x_i\otimes\mathbf u_i\otimes\mathbf v_i
\end{equation}
for a set of orthonormal vectors $\{\mathbf x_i\in\mathbb R^I : 1\leq i\leq r\}$, and unit vectors $\{\mathbf u_i\in \mathbb R^J : 1\leq i\leq r\}$ and $\{\mathbf v_i\in\mathbb R^K : 1\leq i\leq r\}$ satisfy
\[
\langle\mathbf u_i,\mathbf u_j\rangle\langle \mathbf v_i,\mathbf v_j\rangle=\delta_{ij}, \text{for all }i,j=1,\dots,r,
\]
where $\delta_{ij}$ is the Kronecker symbol.  Similar results hold for the matrix flattenings in mode $2$ and $3$. 
\end{corollary}

\begin{proof}
The first part follows from Propositions~\ref{prop:lowerbound} and \ref{prop:svd-nuclear}. The orthonormality follows from those of $\mathbf x_i$'s and $Z_i$'s. 
The remaining follows from \eqref{hs-norm} and the hypothesis that $\|Z_i\|_\ast=1$, which together imply that $Z_i$'s are all rank one matrices. 
\end{proof}

Corollary~\ref{cor:nuclear} gives a
computable criterion for the equivalence between the tensor nuclear norm and its matrix flattening nuclear norm.
If a tensor has the decomposition \eqref{diagonaldec}, then $\|\mathcal A\|_\ast=\|A_{(1)}\|_\ast$ \cite{yuan-zhang}. 

Proposition~\ref{prop:svd-nuclear} has the merit to measure how far the computed matrix flattening nuclear norm from the true nuclear norm of the tensor. It may happens that 
\[
\|\vee Z_i\|_\ast:=\max\{\|Z_i\|_\ast  : 1\leq i\leq r\}
\]
is much larger than most of the $\|Z_i\|_\ast$'s. Therefore, in practical computation, we can determined the accuracy of the nuclear norm of the tensor by the interval
\[
\bigg[\max\big\{\|A_{(i)}\|_\ast : i=1,2,3\big\},\ \ \min\bigg\{\sum_{i=1}^r\sigma_i\|Z_i\|_\ast, \sum_{j=1}^s\mu_j\|S_j\|_\ast, \sum_{k=1}^t\gamma_k\|T_k\|_\ast\bigg\}\bigg],
\]
where $\sum_{j=1}^s\mu_j\|S_j\|_\ast$ is the upper bound given by the matrix flattening in mode $2$, and $\sum_{k=1}^t\gamma_k\|T_k\|_\ast$ is that for mode $3$. 

We arrive at the main theorem. 
\begin{theorem}[The Relation]\label{thm:bound}
For any $3$-tensor $\mathcal A\in\mathbb R^{I\times J\times K}$, we have
\begin{equation}\label{upperbound}
\|A_{(1)}\|_\ast\leq \|\mathcal A\|_\ast\leq \sqrt{\min\{J,K\}}\|A_{(1)}\|_\ast. 
\end{equation}
Moreover, both the bounds on $\|\mathcal A\|_\ast$ are sharp. 
Similar results hold for matrix flattenings in mode $2$ and mode $3$.
\end{theorem}

\begin{proof}
The results follow from Propositions~\ref{prop:lowerbound} and \ref{prop:svd-nuclear}, and the fact that a $J\times K$ matrix of the Hilbert-Schmidt norm (or the Frobenius norm) $1$ can have the nuclear norm at most $\sqrt{\min\{J,K\}}$. 

The sharpness of the left hand side inequality follows when $\min\{J,K\}=1$, in which case the tensor is essentially a matrix and henceforth the inequality becomes an equality. 

For the right hand side inequality, let $I=1$ and tensor 
\[
\mathcal A= \sum_{i=1}^J\frac{1}{\sqrt{J}}1\otimes\mathbf e_{2, i}\otimes\mathbf e_{3,i}, 
\]
where $\{\mathbf e_{2,i} : i=1,\dots,J\}$ is the standard orthonormal basis of $\mathbb R^J$ and $\{\mathbf e_{3,i} : i=1,\dots,J\}$ are  $J$ standard basis vectors in $\mathbb R^K$. Then, it follows that
\[
\|A_{(1)}\|=\|\sum_{i=1}^J\frac{1}{\sqrt{J}}\mathbf e_{2, i}\otimes\mathbf e_{3,i}\|_{\operatorname{HS}}=1,
\]
and
\[
\|\mathcal A\|_\ast=\|\sum_{i=1}^J\frac{1}{\sqrt{J}}\mathbf e_{2, i}\otimes\mathbf e_{3,i}\|_\ast=\sqrt{J}. 
\]
Therefore, we have that the inequality becomes an equality. 
\end{proof}

\begin{corollary}\label{cor:av}
Let $I\leq J\leq K$. For any $3$-tensor $\mathcal A\in\mathbb R^{I\times J\times K}$, we have
\begin{equation*}
\|\mathcal A\|_{\#}\leq \|\mathcal A\|_\ast\leq \sqrt{J}\|\mathcal A\|_{\#}. 
\end{equation*}
\end{corollary}

\section{Generalization}\label{sec:general}
For any positive integers $N\geq 3$, $I_1\leq\dots\leq I_N$ and an $N$-tensor $\mathcal A\in\mathbb R^{I_1\times\dots\times I_N}$,
we can define in a similar fashion the matrix flattenings in mode $1$ up to mode $N$, and
\[
\|\mathcal A\|_{\#}:=\frac{1}{N}\sum_{i=1}^N\|A_{(i)}\|_\ast. 
\]
All the results in the previous sections can be generalized to tensors of higher orders, except the sharpness result of the upper bound in Theorem~\ref{thm:bound} which is unknown and suspected to be most likely false.

To this end, only the next lemma should be outlined. 
\begin{lemma}\label{lem:general}
For any positive integers $N\geq 3$, $I_1\leq\dots\leq I_N$ and an $N$-tensor $\mathcal A\in\mathbb R^{I_1\times\dots\times I_N}$ with $\|\mathcal A\|_{\operatorname{HS}}=1$, we have
\[
\|\mathcal A\|_\ast \leq \sqrt{\prod_{i=1}^{N-1}I_i}.
\]
\end{lemma}

\begin{proof}
Let $\{\mathbf e_{i,j} \in\mathbb R^{I_i} : j=1,\dots,I_i\}$ be the standard orthonormal basis of $\mathbb R^{I_i}$ for $i=1,\dots,N-1$. Then,
\[
\mathcal A=\sum_{1\leq i_j\leq I_j,\ 1\leq j\leq N-1}\mathbf a_{i_1\dots i_{N-1}\cdot}\otimes\mathbf e_{1,i_1}\otimes\dots\otimes\mathbf e_{N-1,i_{N-1}},
\]
where $\mathbf a_{i_1\dots i_{N-1}\cdot}$'s are the mode $N$ vectors of $\mathcal A$. 
Since $\{\mathbf e_{1,i_1}\otimes\dots\otimes\mathbf e_{N-1,i_{N-1}} : 1\leq i_j\leq I_j,\ 1\leq j\leq N-1\}$ is the standard orthonormal basis of $\mathbb R^{I_1\times\dots\times I_{N-1}}$, we have that
\begin{align*}
\|\mathcal A\|_{\operatorname{HS}}^2&=\sum_{1\leq i_j\leq I_j,\ 1\leq j\leq N-1}\langle \mathbf a_{i_1\dots i_{N-1}\cdot}\otimes\mathbf e_{1,i_1}\otimes\dots\otimes\mathbf e_{N-1,i_{N-1}}, \mathbf a_{i_1\dots i_{N-1}\cdot}\otimes\mathbf e_{1,i_1}\otimes\dots\otimes\mathbf e_{N-1,i_{N-1}}\rangle\\
&=\sum_{1\leq i_j\leq I_j,\ 1\leq j\leq N-1}\|\mathbf a_{i_1\dots i_{N-1}\cdot}\|^2.
\end{align*}
On the other side,
\[
\sum_{1\leq i_j\leq I_j,\ 1\leq j\leq N-1}\|\mathbf a_{i_1\dots i_{N-1}\cdot}\otimes\mathbf e_{1,i_1}\otimes\dots\otimes\mathbf e_{N-1,i_{N-1}}\|=\sum_{1\leq i_j\leq I_j,\ 1\leq j\leq N-1}\|\mathbf a_{i_1\dots i_{N-1}\cdot}\|.
\]
Since $\|\mathcal A\|_{\operatorname{HS}}=1$, we have that
\[
\sum_{1\leq i_j\leq I_j,\ 1\leq j\leq N-1}\|\mathbf a_{i_1\dots i_{N-1}\cdot}\|\leq \sqrt{\prod_{i=1}^{N-1}I_i}.
\]
The result on $\|\mathcal A\|_\ast$ then follows from \eqref{decomp}. 
\end{proof}

We then have a similar theorem to Theorem~\ref{thm:bound}. 
\begin{theorem}[General Relation]\label{thm:bound-general}
For any positive integers $N\geq 3$, $I_1\leq\dots\leq I_N$ and an $N$-tensor $\mathcal A\in\mathbb R^{I_1\times\dots\times I_N}$, we have
\[
\|A_{(1)}\|_\ast\leq \|\mathcal A\|_\ast\leq \sqrt{\prod_{i=2}^{N-1}I_i}\|A_{(1)}\|_\ast. 
\]
Therefore,
\[
\|\mathcal A\|_{\#}\leq \|\mathcal A\|_\ast\leq \sqrt{\prod_{i=2}^{N-1}I_i}\|\mathcal A\|_{\#}. 
\]
\end{theorem}
 
\bibliographystyle{model6-names}

\end{document}